\documentclass{article}
\usepackage{graphicx} % Required for inserting images
\usepackage{amsmath}
\usepackage[margin=0.8in]{geometry}
\usepackage{amsthm}
\usepackage{comment}
\usepackage{amssymb}
\usepackage{authblk}
\usepackage{mathtools}
\usepackage[style=alphabetic]{biblatex}
\DeclareMathOperator{\Hom}{Hom}
\DeclareMathOperator{\End}{End}
\DeclareMathOperator{\rank}{rank}
\DeclareMathOperator{\Sym}{Sym}

\newtheorem{lemma}{Lemma}
\newtheorem{theorem}{Theorem}

\newtheorem*{remark}{Remark}
\newtheorem*{theoremA}{Theorem A}
\newtheorem*{theoremB}{Theorem B}

\newcommand\restr[2]{{% we make the whole thing an ordinary symbol
  \left.\kern-\nulldelimiterspace % automatically resize the bar with \right
  #1 % the function
  \vphantom{\big|} % pretend it's a little taller at normal size
  \right|_{#2} % this is the delimiter
  }}

\title{Commuting varieties in bad characteristic}
\author{Vlad Roman}
\date{\today}

\begin{document}

\maketitle

\begin{abstract}
    Let $k$ be an algebraically closed field of characteristic $2$. We consider the commuting variety and the commuting nilpotent variety of the Lie algebra $\mathfrak{sp}_{2n}$, namely the sets $\mathcal{C}_2(\mathfrak{sp}_{2n})=\{ (x,y) \in \mathfrak{sp}_{2n} \times \mathfrak{sp}_{2n} \mid [x,y]=0\}$ and $\mathcal{C}_2^{\text{nil}}(\mathfrak{sp}_{2n})=\{ (x,y) \in \mathfrak{sp}_{2n} \times \mathfrak{sp}_{2n} \mid x,y \text{ nilpotent, } [x,y]=0\}$ and prove that they are both irreducible, of dimensions $\dim(\mathfrak{sp}_{2n}) + 2n$ and $\dim(\mathfrak{sp}_{2n}) + n-1$, respectively.
\end{abstract}

\section{Introduction}
For $k$ an algebraically closed field and $\mathfrak{g}$ a Lie algebra over $k$, we define
$\mathcal{C}_2(\mathfrak{g})=\{ (x,y) \in \mathfrak{g} \times \mathfrak{g} \mid [x,y]=0\}$,
the commuting variety of $\mathfrak{g}$. The investigation of these varieties begins with a paper of Motzkin and Taussky \cite{Motzkin-Taussky}, in which they prove that $\mathcal{C}_{2}(\mathfrak{gl_n})$ is is an irreducible variety of dimension $n^2+n$ (for $k$ of any characteristic). At a later date, Richardson \cite{Richardson} proved that with the assumption that $char(k)=0$ and $\mathfrak{g}$ is reductive, $\mathcal{C}_2(\mathfrak{g})$ is irreducible of dimension $\dim(\mathfrak{g}) + \rank(\mathfrak{g})$, where $\rank(\mathfrak{g})$ is the dimension of a maximal toral subalgebra of $\mathfrak{g}$. Following this, Levy \cite{Levy} extends the result to a field $k$ of good characteristic, that is $char(k) \neq 2$ for corresponding Lie groups of type $B, C, D$, $char(k) \neq 2,3$ for $G_2, F_4, E_6, E_7$ and $char(k) \neq 2, 3, 5$ for $E_8$. There is not much study in the case of bad characteristic. In type $A$, the author \cite{mypaper} proves that in not very good characteristic (meaning $char(k)=p$ with $p \mid n$ for $\mathfrak{gl_n}$), there are two irreducible components, giving the dimension of each component as well. 

Let $k$ be an algebraically closed field of characteristic $2$ and $V$ be a finite-dimensional vector space of dimension $2n$ over $k$. Let $G=Sp(V)$ be the group of transformations which preserve a non-degenerate (skew-)symmetric bilinear form on $V$ and let $\mathfrak{sp}_{2n}=\mathfrak{sp}(V)$ be the Lie algebra of $G$. We investigate the commuting variety of $\mathfrak{sp}_{2n}$. In contrast to the case of good characteristic (i.e. $char(k) \neq 2$ for type $C$), the structure of centralizers in the group of (nilpotent) elements $e \in \mathfrak{sp}(V)$ is more subtle in characteristic $2$. The monograph by Liebeck and Seitz \cite{liebeck-seitz} gives the precise structure of the centralizers $C_G(e)$ and describes the indecomposable summands of $e$, also providing formulas for the dimensions of the centralizers using these indecomposables. However, they do not provide the dimensions of the centralizers in the Lie algebra of the possible indecomposables, which is a key ingredient in the proof of the irreducibility of $\mathcal{C}_2(\mathfrak{sp}_{2n})$. These dimensions (or upper bounds for them) are computed in section \ref{centralizers}. Using these, we obtain the following results:

\begin{theoremA}
    Over an algebraically closed field $k$ of characteristic $2$, $\mathcal{C}_2(\mathfrak{sp}_{2n})$ is an irreducible variety of dimension $\dim(\mathfrak{sp}_{2n}) + 2n$.
\end{theoremA}

\begin{theoremB}
    Over an algebraically closed field $k$ of characteristic $2$, $\mathcal{C}_2^{\text{nil}}(\mathfrak{sp}_{2n})$ is an irreducible variety of dimension $\dim(\mathfrak{sp}_{2n}) + n-1$.
\end{theoremB}

\section{Centralizers in the group vs. centralizers in the Lie algebra}\label{centralizers}
Let $k$ be an algebraically closed field of characteristic $2$. Let $V$ be the natural module on which elements of $\mathfrak{sp}_{2n}$ act. According to \cite{liebeck-seitz}, $V$ decomposes (for $x \in \mathfrak{sp}_{2n}$ nilpotent) as an orthogonal direct sum of indecomposable modules, which are classified in Table $4.1$ in \cite{liebeck-seitz}. Moreover, in Chapter $4$, a formula for the dimension of the centralizer of an element in the corresponding algebraic group is provided, namely, if $x \in \mathfrak{sp}_{2n}$ has Jordan blocks of sizes $t_1 \ge \dots \ge t_r$, then 
$$\dim C_G(x) = \sum_{i=1}^r (it_i - \chi(t_i)),$$
where the function $\chi$ is defined previously in the chapter.

Now decompose $V=U \perp W$, where $W$ is indecomposable and has the smallest Jordan block sizes, so either $W$ is a single Jordan block of size $t_r$ (the $V(t_r)$ in the notation of the book) or $W = W(t_r)$ or $W_{\ell}(t_r)$ and has two Jordan blocks of size $t_r$. We prove the following lemma:

\begin{lemma}\label{centralizer_in_Lie_algebra}
    Let $x \in \mathfrak{sp}_{2n}$ be nilpotent.
    If $W=V(t_r)=J_{t_r}$ is a single Jordan block, then
    $$\dim C_{\mathfrak{sp}_{2n}}(x) = \dim C_{\mathfrak{sp}(U)}(x_U) + \dim C_{\mathfrak{sp}(W)}(x_W) + (r-1)t_r$$ and if $W = W(t_r)$ or $W_{\ell}(t_r)$ (two Jordan blocks of size $t_r$), then 
    $$\dim C_{\mathfrak{sp}_{2n}}(x) = \dim C_{\mathfrak{sp}(U)}(x_U) + \dim C_{\mathfrak{sp}(W)}(x_W) + 2(r-2)t_r.$$
\end{lemma}

\begin{proof}
     The Lie algebra as a module for $x$ decomposes as 
     $$\mathfrak{sp}(V) \cong \mathfrak{sp}(U) \oplus \mathfrak{sp}(W) \oplus (U \otimes W).$$ 
    Indeed, let $\langle \cdot,\cdot \rangle:V \times V \rightarrow k$ be the symplectic form, so we have
    $$\mathfrak{sp}(V) = \{ T \in \operatorname{End}(V) \mid \langle Tv, w \rangle  + \langle v, Tw \rangle  = 0 \ \forall v,w \in V \}$$
    and similarly for $\mathfrak{sp}(U), \mathfrak{sp}(W)$. With respect to $V=U\perp W$, every $T\in\operatorname{End}(V)$ has form
    $$T = \begin{pmatrix} A & B \\ C & D \end{pmatrix}, 
\qquad A:U\to U,\ B:W\to U,\ C:U\to W,\ D:W\to W.$$
    Since the bilinear form is block diagonal $\langle \cdot,\cdot \rangle=\langle \cdot,\cdot \rangle_U \oplus \langle \cdot,\cdot \rangle_W$, the symplectic condition translates into the following: 
For $u,u'\in U$,
  \[
  \langle Au,u'\rangle_U + \langle u,Au' \rangle_U = 0, \text{ meaning } A \in \mathfrak{sp}(U).
  \]
For $w,w'\in W$,
  \[
  \langle Dw,w' \rangle_W + \langle w,Dw' \rangle_W = 0, \text{ meaning }D \in \mathfrak{sp}(W).
  \]
For $u\in U, w\in W$,
  \[
  \langle Bw, u \rangle_U + \langle w, Cu \rangle_W = 0.
  \]
  Defining the adjoint $B^*:U\to W$ by 
  \[
  \langle Bw, u \rangle_U = \langle w, B^*u \rangle_W,
  \]
  this condition yields $C = B^*$. Therefore, $T$ has the form
$$T = \begin{pmatrix} A & B \\ B^* & D \end{pmatrix}, 
\quad A \in \mathfrak{sp}(U),\ D \in \mathfrak{sp}(W),\ B \in \Hom(W,U).$$
Thus, the decomposition (as vector spaces) is
$$\mathfrak{sp}(V) \cong \mathfrak{sp}(U) \oplus \mathfrak{sp}(W) \oplus \Gamma,$$
where $\Gamma \subseteq \Hom(W,U) \oplus \Hom(U^*,W^*)$, so $\Gamma \cong \Hom(W,U)$.
Since the modules are self-dual, we have $U \cong U^*$ and $W \cong W^*$, so we obtain
$$\Hom(W,U) \cong U \otimes W^* \cong U \otimes W.$$

Now write $x=x_U \oplus x_W$ and suppose $y$ centralizes $x$. With respect to the decomposition $V=U \perp W$, any $y \in \mathfrak{sp}(V)$ has form 
$$y=\begin{pmatrix} A & B \\ C & D \end{pmatrix}, 
\qquad A \in \End(U),\ D \in \End(W),\ B \in \Hom(W,U), \ C \in \Hom (U,W).$$
The relation $[y,x]=0$ is thus equivalent to the relations
\[
[A,x_U]=0, \quad [D,x_W]=0, \quad B x_W = x_U B, \quad C x_U = x_W C,
\] so $A \in C_{\End(U)}(x_U)$, $D \in C_{\End(W)}(x_W)$ and $B \in \Hom_x(W,U), C \in \Hom_x(U,W)$.
As before, we have $C=B^*$, so the off-diagonal terms are determined by $B$. Hence,
\[
\dim C_{\mathfrak{sp}(V)}(x)
= \dim C_{\mathfrak{sp}(U)}(x_U)
+ \dim C_{\mathfrak{sp}(W)}(x_W)
+ \dim \mathrm{Hom}_x(W,U).
\]
Recall the standard fact that for two Jordan blocks $J_a,J_b$, we have
\[
\dim \mathrm{Hom}_x(J_a,J_b) = \min(a,b).
\]
If $W = V(t_r) = J_{t_r}$ is a single Jordan block, we have
\[
\dim \mathrm{Hom}_x(W,U)
= \dim \Hom_x(J_{t_r},\bigoplus_{i=1}^{r-1} J_{t_i})
= \sum_{i=1}^{r-1} \dim \Hom_x(J_{t_r},J_{t_i})
= \sum_{i=1}^{r-1} \min(t_r,t_i)
= \sum_{i=1}^{r-1} t_r
= (r-1)t_r.
\]
If $W$ is indecomposable with two blocks of size $t_r$
(either $W(t_r)$ or $W_{\ell}(t_r)$), we have
$$
\dim \mathrm{Hom}_x(W,U)
= \dim \Hom_x(J_{t_r} \oplus J_{t_r},\bigoplus_{i=1}^{r-2} J_{t_i})
= \sum_{i=1}^{r-2} \dim \Hom_x(J_{t_r} \oplus J_{t_r},J_{t_i})
= \sum_{i=1}^{r-2} 2 \dim \Hom_x(J_{t_r},J_{t_i}) $$

$$= 2 \sum_{i=1}^{r-2} \min(t_r,t_i)
= 2 \sum_{i=1}^{r-2} t_r
= 2(r-2)t_r
$$
and the result follows.
\end{proof}

\begin{lemma}\label{centralizer_in_group}
    Let $x \in \mathfrak{sp}_{2n}$ be nilpotent. If $W=V(t_r)=J_{t_r}$ is a single Jordan block, then $$\dim C_{Sp_{2n}}(x) = \dim C_{Sp_{(U)}}(x_U) + \dim C_{Sp_{(W)}}(x_W) + (r-1)t_r$$ and if $W = W(t_r)$ or $W_{\ell}(t_r)$ (two Jordan blocks of size $t_r$), then $$\dim C_{Sp_{2n}}(x) = \dim C_{Sp_{(U)}}(x_U) + \dim C_{Sp_{(W)}}(x_W) + 2(r-2)t_r.$$
    
\end{lemma}

\begin{proof}
    If $W=V(t_r)$ is a single Jordan block of size $t_r$ we have
    \begin{align*}
    \dim C_G(x) &= \sum_{i=1}^{r} (it_i - \chi(t_i)) \\
    &= \sum_{i=1}^{r-1} it_i - \sum_{i=1}^{r} \chi(t_i) + rt_r \\
    &=\sum_{i=1}^{r-1} it_i - \sum_{i=1}^{r} \chi(t_i) + (r-1)t_r + t_r \\
    &= \dim C_{Sp(U)}(x_U) + \dim C_{Sp(W)}(x_W) + (r-1)t_r.
    \end{align*}

    If $W = W(t_r)$ or $W_{\ell}(t_r)$ (two Jordan blocks of size $t_r$) we have
    \begin{align*}
    \dim C_G(x) &= \sum_{i=1}^{r} (it_i - \chi(t_i)) \\
    &= \sum_{i=1}^{r-2} it_i - \sum_{i=1}^{r} \chi(t_i) + (r-1)t_{r-1}+rt_t \\
    &=\sum_{i=1}^{r-2} it_i - \sum_{i=1}^{r} \chi(t_i) + (2r-1)t_r \\
    &=\sum_{i=1}^{r-2} it_i - \sum_{i=1}^{r} \chi(t_i) + 3t_r+ (2r-4)t_r  \\
    &= \dim C_{Sp(U)}(x_U) + \dim C_{Sp(W)}(x_W) + 2(r-2)t_r.
    \end{align*}
In the above computations we have used the dimensions of the centralizers in $G$ of the indecomposables (proved in \cite{liebeck-seitz}). These are also shown later in the proof of Lemma \ref{discrepancy}.
\end{proof}
    
The next lemma is essential in proving the irreducibility of the varieties.
\begin{lemma}\label{discrepancy}
   Let $x \in \mathfrak{sp}_{2n}$ be a nilpotent element. Define the \textbf{discrepancy} of $x$ as
   $$\Delta(x) =\dim C_{\mathfrak{sp}_{2n}}(x) - \dim C_{Sp_{2n}}(x).$$
   Then we have $\Delta(x) \leq n$, with equality precisely when all the indecomposables are single Jordan blocks (the $V(t_i)$).
\end{lemma}

\begin{proof}
    According to \cite{liebeck-seitz}, $V$ decomposes as a direct sum $V\cong \bigoplus_j M_j$, where each $M_j$ is one of the three possible indecomposables, namely $V(t_j),W(t_j)$ or $W_{\ell}(t_j)$. In light of this and Lemmas \ref{centralizer_in_Lie_algebra} and \ref{centralizer_in_group}, it follows that the discrepancy is equal to the sum of the discrepancies on each indecomposable, that is
    $$\Delta(x)=\sum_j \bigg( \dim C_{\mathfrak{sp}(M_j)}(x_{M_j}) - \dim C_{Sp(M_j)}(x_{M_j}) \bigg),$$
    where $x=\bigoplus_j x_{M_j}$. Hence, the problem reduces to computing the discrepancy for each one of the three indecomposables.

\subsection*{First computation: Centralizers in the group}
    
    We use the formulas in \cite{liebeck-seitz} to compute the dimension of the centralizer in the group first.
    
    Case $1$: $V\downarrow x=V(2n)$ and $x=J_{2n}$, a single Jordan block of dimension $2n$ and $l=2n/2=n$. We have
    \begin{align*}
    \dim C_G(x) &= 2n-\chi(2n) \\
    &= 2n-max\{0,min\{ 2n-2n+l,l \}\} \\
    &=2n-max\{ 0,n \} \\
    &= n.
\end{align*}
Note that this is to be expected since a single Jordan block is regular nilpotent, so its centralizer has dimension equal to the rank of the Lie algebra.

    Case $2$: $V\downarrow x=W(n)$ and $x=J_{n}\oplus J_n$, two Jordan blocks of dimension $n$ and $l=0$. Here $t_1=t_2=n$. We have
    \begin{align*}
    \dim C_G(x) &= \sum_{i=1}^2(it_i-\chi(t_i)) \\
    &= n-\chi(n)+2n-\chi(n) \\
    &= 3n-2 \chi(n) \\
    &= 3n-max\{0,min\{ n-n+0,0 \}\} \\
    &= 3n.
\end{align*}

    Case $3$: $V\downarrow x=W_{\ell}(n)$ and $x$ is of the form $\begin{pmatrix}
A&B\\ 0&A^T
\end{pmatrix}$ with $B=B^T$ and $A=J_n$. Here $0 < l < n/2$ and $t_1=t_2=n$. We have
    \begin{align*}
    \dim C_G(x) &= \sum_{i=1}^2(it_i-\chi(t_i)) \\
    &= n-\chi(n)+2n-\chi(n) \\
    &= 3n-2 \chi(n) \\
    &= 3n-2 max\{0,min\{ n-n+l,l \}\} \\
    &=3n-2max\{ 0,l \} \\
    &= 3n-2l.
\end{align*}

\subsection*{Second computation: Centralizers in the Lie algebra}

Here we compute the dimension of the centralizer in the Lie algebra for  two of the indecomposables and obtain an upper bound for the third one.

Case $1$: $V\downarrow x=V(2n)$ and $x=J_{2n}$, a single Jordan block of dimension $2n$. We may choose $\{ e_1,\dots,e_{2n} \}$ a cyclic basis for $V$, so that $xe_i=e_{i+1}$ for $i < 2n$ and $xe_{2n}=0$. Moreover, we can arrange the symplectic form in antidiagonal form, so we have $\langle e_i,e_j\rangle = \delta_{i+j,2n+1}$. On one hand, since $x$ is regular nilpotent, its full centralizer consists of polynomials in $x$, so
$C_{\End(V)}(x)=\{ p(x) \mid p \in k[X],\deg p \leq 2n-1\}$
and $\dim C_{\End(V)}(x)=2n$. On the other hand, we claim that any polynomial in $x$ lies in $\mathfrak{sp}(V)$. Extending afterwards by linearity, it suffices to check the symplectic condition for monomials $x^k$. We have
$$\langle x^k e_i,e_j \rangle + \langle e_i,x^k e_j \rangle
= \langle e_{i+k},e_j \rangle + \langle e_i, e_{j+k} \rangle
= \delta_{i+k+j,2n+1} + \delta_{i+j+k,2n+1}=2 \delta_{i+k+j,2n+1} =0.$$
Thus, $\langle x^k v,w \rangle + \langle v,x^k w \rangle =0$ for all $v,w\in V$. Hence, $C_{\End(V)}(x) \subseteq \mathfrak{sp}(V)$ and since $C_{\mathfrak{sp}(V)}(x) = C_{\End(V)}(x) \cap \mathfrak{sp}(V)$, we conclude that $\dim C_{\mathfrak{sp}_{2n}}(x)=\dim C_{\End(V)}(x) = 2n$.

Case $2$: $V\downarrow x=W(n)$ and $x=J_{n}\oplus J_n$ two Jordan blocks of dimension $n$. Write $V=V_1 \oplus V_2$ with $\dim(V_1)=\dim(V_2)=n$ and let $x=x_1\oplus x_2$ where each $x_i$ is a single Jordan block of size $n$ acting on $V_i$. Suppose $y \in \End(V)$ commutes with $x$ and write 
$$y = \begin{pmatrix} y_{11} & y_{12} \\ y_{21} & y_{22} \end{pmatrix}, 
\qquad y_{ij}:V_j\to V_i.$$
The condition $[y,x]=0$ yields $y_{ij}x_j=x_iy_{ij}$, $i=1,2$. Since each $x_i$ is a single Jordan blocks, its commutator consists of polynomials in it, so $y_{11} \in k[x_1]$, $y_{22} \in k[x_2]$, $y_{12} \in \Hom_{k[x]}(V_2,V_1)$, $y_{21} \in \Hom_{k[x]}(V_1,V_2)$. Now,
$\dim \Hom_{k[x]}(V_2,V_1) = \dim \Hom_{k[x]}(J_n,J_n)=n=\dim \Hom_{k[x]}(V_1,V_2)$. Therefore, as a vector space, we have $\dim \{ y \in \End(V) \mid [y,x]=0\} = n+n+n+n=4n$. Now impose the symplectic condition $\langle yv,w \rangle +\langle v,yw \rangle=0$ for all $v,w \in V$. Let us analyse this blockwise. Denote the form by 
$$\Omega = \begin{pmatrix} 0 & I_n \\ I_n & 0 \end{pmatrix}.$$
Then $y \in \mathfrak{sp}(V)$ if and only if $y^t \Omega + \Omega Y=0$. This yields $y_{21}^t=y_{21}$, $y_{11}^t=y_{22}$, $y_{22}^t=y_{11}$, $y_{12}^t=y_{12}$. Therefore, $y_{22}$ is determined by $y_{11}$ and both $y_{12}$ and $y_{21}$ are (skew-)symmetric $n \times n$ matrices. Counting dimensions, we obtain $\dim C_{\mathfrak{sp}_{2n}}(x)=3n$ (one $n$ comes from the $y_{11}$ part (and the $y_{22}$ part is determined by $y_{11}$, so no extra dimension is gained) and one factor of $n$ comes from each of $y_{12}$ and $y_{21}$ (these commute with the corresponding Jordan blocks, so they are polynomials in the blocks and moreover, they are symmetric, but note that this condition does not reduce the dimension of the space, so we obtain $n$)).

Case $3$: $V\downarrow x=W_{\ell}(n)$ and $x$ is of the form $\begin{pmatrix}
A&B\\ 0&A^T
\end{pmatrix}$ with $B=B^T$ and $A=J_n$. Let $V=L\oplus L^{*}$ with $\dim L=n$ and use the symplectic form with Gram matrix as in the previous case.
Then
\[
\mathfrak{sp}(V)=\left\{
X=\begin{pmatrix}X_{11}&X_{12}\\ X_{21}&X_{22}\end{pmatrix} \mid\;
X_{12},X_{21} \text{ symmetric},\; X_{22}=X_{11}^{T}
\right\}.
\]
Relative to the parabolic $P=\operatorname{Stab}(L)$ we have the decomposition (as vector spaces)
\[
\mathfrak{sp}(V)=
\underbrace{\left\{\begin{pmatrix}0&0\\ T&0\end{pmatrix}\right\}}_{\mathfrak u^{-}\cong \operatorname{Sym}^2(L)}
\oplus
\underbrace{\left\{\begin{pmatrix}Y&0\\ 0&Y^{T}\end{pmatrix}\right\}}_{\mathfrak m\cong \mathfrak{gl}(L)}
\oplus
\underbrace{\left\{\begin{pmatrix}0&S\\ 0&0\end{pmatrix}\right\}}_{\mathfrak u\cong \operatorname{Sym}^2(L^{*})},
\]
with $S,T$ symmetric and $Y$ arbitrary. For an indecomposable of type $W_\ell(n)$ (call it $x$) one can choose $L$ so that $L$ is $x$-stable and
\[
x=\begin{pmatrix} A & B\\ 0 & A^{T} \end{pmatrix},\qquad
A=J_n,\quad B=B^{T}.
\]
Here $A$ is a single Jordan block on $L$.

We bound $\dim \mathfrak g^{x}=\dim \ker(\operatorname{ad} x)$ by analyzing each block.

\subsubsection*{Levi part $\mathfrak m$}
Let $X\in\mathfrak m$, $X=\begin{pmatrix}Y&0\\0&Y^T\end{pmatrix}$. Then
\[
[x,X]
=\begin{pmatrix}
[A,Y] & BY^{T}-YB\\
0 & [A^T,Y^T]
\end{pmatrix}.
\]
So $X\in \mathfrak g^{x}$ if and only if
\[
[A,Y]=0,\qquad BY^{T}=YB.
\]
The centralizer of $A=J_n$ in $\mathfrak{gl}_n$ is $\{f(A): f\in k[t], \deg f<n\}$, hence it has dimension $n$.  
The second condition only reduces dimension. Thus
\[
\dim\ker(\operatorname{ad}x|_{\mathfrak m}) \le n.
\]

\subsubsection*{Upper part $\mathfrak u$}
Take $X=\begin{pmatrix}0&S\\0&0\end{pmatrix}$ with $S=S^T$. Then
\[
[x,X]=\begin{pmatrix}0&AS+SA^T\\0&0\end{pmatrix}.
\]
So the condition is
\[
AS+SA^T=0.
\]
Writing $S=(s_{ij})$ with $s_{ij}=s_{ji}$, for $A=J_n$ this gives
\[
s_{i+1,j}=s_{i,j+1}\quad (1\le i,j\le n-1).
\]
This recursion shows all entries are determined by the first row $(s_{11},\dots,s_{1n})$, hence
\[
\dim\ker(\operatorname{ad}x|_{\mathfrak u})\le n.
\]

\subsubsection*{Lower part $\mathfrak u^{-}$}
Take $X=\begin{pmatrix}0&0\\T&0\end{pmatrix}$ with $T=T^T$. Then
\[
[x,X]=\begin{pmatrix}BT&0\\ A^T T+TA& TB\end{pmatrix}.
\]
The conditions are
\[
A^T T+TA=0,\qquad BT=0,\qquad TB=0.
\]
The equation $A^TT+TA=0$ is the same recursion as the one for $S$, hence the centralizer is at most $n$-dimensional. The conditions $BT=TB=0$ are further linear restrictions, so
\[
\dim\ker(\operatorname{ad}x|_{\mathfrak u^{-}})\le n.
\]
Summing up, we have
\[
\dim C_{\mathfrak{sp}_{2n}}(x) = \dim \mathfrak g^{x} \;=\;
\dim\ker(\operatorname{ad}x|_{\mathfrak u^{-}})
+\dim\ker(\operatorname{ad}x|_{\mathfrak m})
+\dim\ker(\operatorname{ad}x|_{\mathfrak u})
\;\le\; n+n+n=3n.
\]

\subsection*{Conclusion}

\begin{center}
\begin{tabular}{|c|c|c|c|}
\hline
$x$           & $\dim C_{\mathfrak{sp}_{2n}}(x)$ & $\dim C_{Sp_{2n}}(x)$ & $\Delta(x)$ \\ \hline
$V(2n)$       & $2n$            & $n$              &    $n$      \\ \hline
$W(n)$        &   $3n$           &          $3n$          &      $0$     \\ \hline
$W_{\ell}(n)$ &   $\le 3n$               &       $3n-2l$        &      $\le 2l<n$      \\ \hline
\end{tabular}
\end{center}
To finish the proof, let us write $$V=\bigoplus_{i} V(2a_i) \oplus \bigoplus_{j}W(m_j) \oplus \bigoplus_{k} W_{\ell_k}(n_k),$$
where $a_i \geq 1,m_j \geq 1,1 \le {\ell_k} < \frac{n_k}{2}$. Then according to the above table, the discrepancy of $x$ is given by
$$\Delta(x) = \sum_i a_i + \sum_j0 + \sum_{k} \Delta(W_{\ell_k}(n_k)) \le \sum_i a_i + \sum_{k} 2\ell_k.$$
On the other hand, the total dimension of the space $V$ is given by
$$\sum_i 2a_i + \sum_j 2m_j + \sum_{k} 2n_{k} = 2n$$
with each $0<\ell_k<\frac{n_k}{2}$, so $2\ell_k<n_k$. Hence,
$$\Delta(x)=\sum_i a_i + \sum_{k} n_k = \frac{1}{2}\bigg( \sum_{i} 2a_i + \sum_{k} 2n_k \bigg) \le \frac{1}{2}\bigg( \sum_{i} 2a_i + \sum_{j} 2m_i + \sum_{k} 2n_k \bigg) = \frac{1}{2}2n=n.$$
Moreover, equality $\Delta(x)=n$ forces equality at every step above. Since each strict inequality $2\ell_k<n_k$ would make the sum strictly smaller, we must have no $W_{\ell_k}(n_k)$ summands. Since dropping the non-negative $\sum_j 2m_j$ also weakens the bound, equality requires no $W(m_j)$ summands either. Hence $V$ is a direct sum of  single Jordan blocks $V(2a_i)$. This completes the proof.
\end{proof}

\section{Commuting variety}
In this section we consider the commuting variety  of $\mathfrak{g}=\mathfrak{sp}_{2n}$, that is the set $$\mathcal{C}_2(\mathfrak{sp}_{2n})=\{ (x,y) \in \mathfrak{sp}_{2n} \times \mathfrak{sp}_{2n} \mid [x,y]=0\}.$$
We immediately obtain a lower bound for the commuting variety:

\begin{lemma}\label{lowerbound}
    Every component of $\mathcal{C}_2(\mathfrak{sp}_{2n})$ has dimension at least $\dim(\mathfrak{sp}_{2n}) + 2n$.
\end{lemma}

\begin{proof}
    We view $\mathfrak{sp}_{2n}$ as the set of self-adjoint operators with respect to the alternating form defining the corresponding algebraic group $Sp_{2n}$. We note that the derived Lie algebra of $\mathfrak{sp}_{2n}$ is contained in $\mathfrak{so}_{2n}$ in characteristic $2$. Indeed, letting $A,B \in \mathfrak{sp}_{2n}$, they are self-adjoint, so $\langle Av,w \rangle = \langle v,Aw \rangle$ for all $v,w \in V$ and the same for $B$. Then 
$$\langle [A,B]v,v \rangle=\langle ABv,v \rangle + \langle BAv,v \rangle =\langle Bv,Av \rangle +\langle Av,Bv \rangle = 2\langle Av,Bv \rangle=0$$
since $\langle \cdot,\cdot \rangle$ is alternating and in characteristic $2$ we have $\langle v,w \rangle= \langle w,v \rangle$. So any commutator is skew-adjoint and this shows that $[\mathfrak{sp}_{2n},\mathfrak{sp}_{2n}] \subseteq \mathfrak{so}_{2n}$. Therefore, considering the morphism $\phi: \mathfrak{sp}_{2n} \times \mathfrak{sp}_{2n} \rightarrow \mathfrak{so}_{2n}$, $(x,y) \mapsto [x,y]$, by the fibre dimension theorem every irreducible component of every fibre of $\phi$ (in particular $\phi^{-1}(0)=\mathcal{C}_2(\mathfrak{sp}_{2n})$) has dimension at least
$$\dim(\mathfrak{sp}_{2n} \times \mathfrak{sp}_{2n})- \dim(\mathfrak{so}_{2n}) = 2(n(2n+1))-n(2n-1)=(2n^2+n)+2n=\dim(\mathfrak{sp}_{2n})+2n.$$
\end{proof}

For the next step, inside $\mathfrak{sp}_2$, pick a $2$-dimensional abelian subalgebra $\mathfrak{a}_0$. Embed $n$ copies of $\mathfrak{a}_0$ to obtain $\mathfrak{a} = \mathfrak{a}_0^{\oplus n} \subseteq \mathfrak{sp}_{2n}$, so we have $\dim \mathfrak{a} = 2n$. Let $G=Sp_{2n}$ and consider $G \cdot (\mathfrak{a} \times \mathfrak{a})$, the set of all $G$-conjugates of pairs from $\mathfrak{a} \times \mathfrak{a}$.

\begin{theorem}
    $\mathcal{C}_2(\mathfrak{sp}_{2n}) = \overline{G \cdot (\mathfrak{a} \times \mathfrak{a})}$ (the closure of $G \cdot (\mathfrak{a} \times \mathfrak{a})$).
\end{theorem}

Write $V=\bigoplus_{i=1}^n V_i$, where each $V_i$ is a symplectic space and $\dim (V_i)=2$. Now, inside each factor $\mathfrak{sp}(V_i)$, the centralizer of a generic pair $(u_i,v_i) \in \mathfrak{a}_0 \times \mathfrak{a}_0$ is precisely $\mathfrak{a}_0$. Indeed, $\mathfrak{a}_0$
is abelian and any $x \in \mathfrak{sp}(V_i)$ that commutes with two linearly independent elements in $\mathfrak{a}_0$ must commute with all of $\mathfrak{a}_0$. Thus, $x \in \mathfrak{a}_0$ since $\dim \mathfrak{sp}(V_i)=3$ and $\mathfrak{a}_0$ is already a maximal commutative Lie subalgebra in this $3$-dimensional Lie algebra. For the global generic centralizer, consider a pair $(u,v) \in \mathfrak{a} \times \mathfrak{a}$ with components $(u_i,v_i) \in \mathfrak{a}_0 \times \mathfrak{a}_0$, all chosen generically. Since $\mathfrak{a}$ is block diagonal, we have $$C_{\mathfrak{g}} (u,v) = \bigoplus_{i=1}^n C_{\mathfrak{sp}(V_i)} (u_i,v_i) = \bigoplus_{i=1}^n \mathfrak{a}_0 = \mathfrak{a}.$$
Let $\mathfrak{a}_{reg} \subseteq \mathfrak{a}$ be the nonempty Zariski open subset on which the above holds. Define $U = G \cdot (\mathfrak{a}_{reg} \times \mathfrak{a}_{reg}) \subseteq \mathcal{C}_2(\mathfrak{sp}_{2n})$. Now consider the action map $\phi : G \times (\mathfrak{a}_{reg} \times \mathfrak{a}_{reg}) \twoheadrightarrow  U \subseteq \mathcal{C}_2(\mathfrak{sp}_{2n})$. Using the fibre dimension theorem, we obtain
$$\dim U = \dim G + \dim (\mathfrak{a}_{reg} \times \mathfrak{a}_{reg}) - \dim(\text{generic centralizer}) = \dim \mathfrak{g} + 4n -2n = \dim(\mathfrak{sp}_{2n})+2n.$$
Hence, $\overline{U}=\overline{G \cdot (\mathfrak{a} \times \mathfrak{a})}$ is an irreducible subvariety of $\mathcal{C}_2(\mathfrak{sp}_{2n})$ (by being the image of an irreducible set) of dimension $\dim(\mathfrak{sp}_{2n})+2n$. Since this matches the lower bound proved in Lemma \ref{lowerbound}, $\overline{U}$ is the candidate dense component of $\mathcal{C}_2(\mathfrak{sp}_{2n})$. Indeed, in the following we show that a component of smaller dimension cannot exist.

Consider a pair $(x,y) \in \mathcal{C}_2(\mathfrak{sp}_{2n})$. Using the Jordan-Chevalley decomposition write $x=x_s+x_n$ and $y=y_s+y_n$ with $x_s,y_s$ semisimple, $x_n,y_n$ nilpotent and $[x_s,x_n]=[y_s,y_n]=0$. Now suppose $x_s$ is non-central (so $x_s$ is not a multiple of the identity, since $Z(\mathfrak{sp}_{2n})=k I$ in characteristic $2$). Then (again by Jordan-Chevalley) there exists a polynomial $f \in k[T]$ such that $x_s=f(x)$. So for any $z$ with $[x,z]=0$, it follows that $z$ commutes with any power of $x$, so by linearity it commutes with any polynomial in $x$. Thus,  $[x_s,z]=[f(x),z]=0$. This shows that $C_{\mathfrak{sp}_{2n}}(x) \subseteq C_{\mathfrak{sp}_{2n}}(x_s)$, so the pair $(x,C_{\mathfrak{sp}_{2n}}(x)) \subseteq (x,C_{\mathfrak{sp}_{2n}}(x_s))$.

Claim: $x_s$ has at least $2$ distinct eigenvalues.

Indeed, suppose for a contradiction that $x_s$ has only one eigenvalue. Since $x_s$ is semisimple (so diagonalizable) and it has a single eigenvalue, it must be a scalar, so $x_s=\lambda I$ for some $\lambda \in k$. In characteristic $2$ every multiple of the identity lies in $\mathfrak{sp}(V)$ and is central. Indeed, if $X=\lambda I$ then for all $v,w \in V$ we have
$$\langle Xv,w \rangle + \langle v,Xw \rangle = \lambda \langle v,w \rangle + \lambda \langle v,w \rangle = 2 \lambda \langle v,w \rangle = 0,$$
so $X \in \mathfrak{sp}(V)$ and $[\lambda I,y]=0$ for all $y \in \mathfrak{sp}(V)$. Thus $x_s$ would be central, a contradiction.

Write the eigenspace decomposition for $x_s$ as $V \cong \bigoplus_{\lambda} V_{\lambda}$, where $\restr{x_s}{V_\lambda}=\lambda \cdot id$. We note that distinct eigenspaces are orthogonal. Indeed, suppose $v \in V_\lambda, w\in V_\mu$ we have
$$0=\langle x_s v,w \rangle + \langle v,x_s w \rangle = \lambda \langle v,w \rangle + \mu \langle v,w \rangle = (\lambda+\mu) \langle v,w \rangle,$$
so if $\lambda \neq \mu$ (note $-\mu=\mu$), then $\lambda+\mu \neq 0$ and thus $\langle v,w\rangle=0$. Hence, $\langle V_\lambda,V_\mu \rangle =0$ for $\lambda \neq \mu$.

Let us also note that we have nondegeneracy on each eigenspace. Let $\lambda$ be an eigenvalue of $x_s$ and set $U=V_\lambda$ and $W=\bigoplus_{\mu \neq \lambda} V_\mu$. By the previous argument we have $V=U \perp W$. If $u \in U$ with $\langle u,U\rangle=0$ then we also have $\langle u,W\rangle=0$, so $\langle u,V\rangle=0$. By nondegeneracy on $V$, it follows that $u=0$. So the form restricts nondegenerately to both $U$ and $W$ and we have an orthogonal decomposition of symplectic spaces $V=U\perp W$.

\begin{comment}
    Now, since $x_s$ is semisimple, its full centralizer in $\mathfrak{gl}(V)$ is block-diagonal with respect to the eigenspace decomposition and so $$C_{\mathfrak{gl}(V)}(x_s)=\bigoplus_\lambda \End(V_\lambda).$$
Intersecting with $\mathfrak{sp}(V)$ and using $V=U\perp W$ with each block symplectic, we obtain
$$C_{\mathfrak{sp}(V)}(x_s) = C_{\mathfrak{gl}(V)}(x_s) \cap \mathfrak{sp}(V) \subseteq \mathfrak{sp}(U) \oplus \mathfrak{sp}(W) \subseteq \mathfrak{sp}(V) = \mathfrak{sp}_{2n}.$$
\end{comment}

Case $1$: Assume $y=y_s+y_n$ with $y_s$ non-central.

Choose an eigenspace decomposition for $y_s$, $V \cong \bigoplus_\lambda V_\lambda$ with each $V_\lambda$ non-degenerate. Since $C_\mathfrak{g}(y) \subseteq C_\mathfrak{g}(y_s) = \bigoplus_\lambda \mathfrak{sp}(V_\lambda)$, any element that commutes with $y$ preserves every $V_\lambda$. Since $y_s$ is non-central, it has at least $2$ distinct eigenvalues, so we may arrange the $V_\lambda$ into two proper non-degenerate summands, so $V=W_1 \oplus W_2$ with the $W_i$ orthogonal, proper, non-degenerate. Then 
$$C_\mathfrak{g}(y) \subseteq \mathfrak{sp}(W_1) \oplus \mathfrak{sp}(W_2), C_G(y) \subseteq Sp(W_1) \times Sp(W_2).$$
Let $\mathfrak{a}_i=\mathfrak{a} \cap \mathfrak{sp}(W_i)$.
By the $2$-dimensional analysis and generic centralizer calculation previously, there exist open subsets $\mathfrak{a}_i^{reg} \subset \mathfrak{a_i}$ such that for $(u_i,v_i) \in \mathfrak{a}_i^{reg} \times \mathfrak{a}_i^{reg}$ we have $C_{\mathfrak{sp}(W_i)} = \mathfrak{a}_i$. Now, by induction on $\dim(V)$ (with base case $\dim(V)=2$), the pair $(x,y)$ restricted to $W_i$ lies in $\overline{(\mathfrak{a}_i \times \mathfrak{a}_i)^{Sp(W_i)}}$. Therefore, we have $(x,y) \in \overline{G \cdot (\mathfrak{a} \times \mathfrak{a})}$.

Case $2$: Assume $x$ commutes with a non-central semisimple element $z$.

Without loss of generality, we may assume that $x_s$ and $y_s$ are scalars (indeed, subtract scalars if necessary, commutativity is still preserved). Then $(x,y) \in C_\mathfrak{g}(y) \times \{ y\}$. Consider the set
$y_{\mu,\lambda}=\mu y + \lambda z$ for $\mu,\lambda \in k$. Since $z$ is semisimple and non-central, for all but finitely many $(\mu,\lambda)$ the semisimple part is non-central (indeed, generically it has at least two eigenvalues on some $2$-dimensional summand). Moreover,
$$[x,y_{\mu,\lambda}] = \mu [x,y] + \lambda [x,z]=0.$$
Thus each generic pair $(x,y_{\mu,\lambda})$ is still in the commuting variety, so by Case $1$, $(x,y_{\mu,\lambda}) \in  \overline{G \cdot (\mathfrak{a} \times \mathfrak{a})}$ for generic $(\mu,\lambda)$. Letting $(\mu,\lambda) \rightarrow (1,0)$, we have $(x,y) \in \overline{G \cdot (\mathfrak{a} \times \mathfrak{a})}$.

Case $3$: Suppose neither $x$ nor $y$ commutes with a non-central semisimple element. Then their semisimple parts are scalars, so write
$$x=\lambda I + x_n, y=\mu I+y_n$$ with $x_n,y_n$ nilpotent and $[x_n,y_n]=0$. Let $Y$ be a component of the commuting variety that contains $(x,y)$. Consider the projection $\pi_1:\mathcal{C}_2(\mathfrak{sp}_{2n}) \twoheadrightarrow \mathfrak{sp}_{2n}$. Then $\pi_1(Y) \subseteq (nilpotents+scalars)$.
Suppose $x_n$ is in a given conjugacy class $C$. Consider the morphism 
$$\phi_C: k \times C \rightarrow \mathfrak{sp}_{2n}$$ given by $$(\lambda,u) \mapsto \lambda I +u.$$
Let $S_C=\{ \lambda I+u \mid \lambda \in k, u \in C\} = \phi_C(k \times C)$ be the image. We have $\dim(k\times C)=1+\dim(C)$ and so $\dim(S_C) \leq 1+\dim(C)$. Since there are finitley many nilpotent orbits,
$$\pi_1(Y) \subseteq \bigcup_C S_C$$ and therefore
$$\dim \pi_1(Y) \leq \max_{C} \dim (S_C) \leq \max_C (1+\dim(C)).$$
In particular, for the orbit containing $x_n$ we have $\dim \pi_1(Y) \leq 1+\dim(C)$.

Next, restrict the projection $\pi_1$ to $Y$. For any $x' \in \mathfrak{g}$ we have $$\pi_1^{-1}(x')=\{ (x',y') \mid [x',y']=0\} = C_{\mathfrak{g}}(x').$$
So for $x'=\lambda I+x_n \in S_C$, $\dim(\text{full fibre over } x')=\dim C_{\mathfrak{g}}(x') = \dim C_{\mathfrak{g}}(x_n)$. The fibre of ${\pi_1}_Y$ over $x'$ is $Y \cap (\{x'\} \times \mathfrak{g})$, so it's a closed subvariety of the full fibre and thus
$$\dim(\text{fibre of } {\pi_1}_Y)\leq \dim C_\mathfrak{g}(x_n).$$
By the fibre-dimension theorem, for a dominant map $Y \rightarrow \pi_1(Y)$, there is an open subset of $\pi_1(Y)$ over which fibres have minimal dimension and in particular we obtain
$$\dim(Y)\leq \dim \pi_1(Y) + \dim (\text{generic fibre of } {\pi_1}_Y) \leq \dim  \pi_1(Y) + \dim C_\mathfrak{g}(x_n).$$
Combining with the previous inequality, we obtain
 \begin{align*}
    \dim(Y) &\leq 1+\dim(C)+\dim C_\mathfrak{g}(x_n) \\
    &= 1+\dim(C)+ \dim C_G(x_n) + \Delta(x_n) \\
    &= 1+\Delta(x_n) + \dim(\mathfrak{sp}_{2n}) \\
    &\leq 1+n+\dim(\mathfrak{sp}_{2n}).
\end{align*}
Finally, together with Lemma \ref{lowerbound}, we conclude that $\mathcal{C}_2(\mathfrak{sp}_{2n})$ is irreducible, its only component being $\overline{G \cdot (\mathfrak{a} \times \mathfrak{a})}$.

\begin{remark}
    In the case $n=1$, we have $\mathfrak{sp}_2 \cong \mathfrak{sl}_2$, so the irreducibility of $\mathcal{C}_2(\mathfrak{sp}_{2})$ follows from \cite{mypaper}.
\end{remark}

\section{Commuting nilpotent variety}
In this section we consider the commuting nilpotent variety of $\mathfrak{sp}_{2n}$, namely the set $$\mathcal{C}_2^{\text{nil}}(\mathfrak{sp}_{2n})=\{ (x,y) \in \mathfrak{sp}_{2n} \times \mathfrak{sp}_{2n} \mid x,y \text{ nilpotent, } [x,y]=0\}.$$
First let us consider the parabolic Lie algebra
$$\mathfrak{p}=\biggl\{ \begin{pmatrix} A & B \\ 0 & A^t \end{pmatrix} \Biggl\vert A \in \mathfrak{gl}_n, B \in \Sym_n \biggr\} \subseteq \mathfrak{sp}_{2n}$$
and its commuting nilpotent variety, $\mathcal{C}_2^{\text{nil}}(\mathfrak{p})=\{ (x,y) \in \mathfrak{p} \times \mathfrak{p} \mid x,y \text{ nilpotent, } [x,y]=0\}$. Notice first that an element $\begin{pmatrix} A & B \\ 0 & A^t \end{pmatrix} \in \mathfrak{p}$ is nilpotent if and only if $A$ is nilpotent. Consider a pair of commuting elements in $\mathfrak{p}$,
$$x=\begin{pmatrix} A_1 & B_1 \\ 0 & A_1^t \end{pmatrix}, y=\begin{pmatrix} A_2 & B_2 \\ 0 & A_2^t \end{pmatrix},$$
so we have
$$[x,y] = \begin{pmatrix} A_1A_2-A_2A_1 & A_1B_2+B_1A_2^t-A_2B_1-B_2A_1^t \\ 0 & A_1^1A_2^t-A_2^tA_1^t \end{pmatrix} = \begin{pmatrix} [A_1,A_2] & A_1B_2+B_1A_2^t+A_2B_1+B_2A_1^t \\ 0 & [A_1^t,A_2^t] \end{pmatrix}.$$
Define a map $F_{A_1,A_2}:\Sym_n \times \Sym_n \rightarrow \Sym_n$ by 
$$F_{A_1,A_2}(B_1,B_2) = A_1B_2+B_1A_2^t+A_2B_1+B_2A_1^t.$$
Note that indeed the image lies in $\Sym_n$ since $A_1B_2+B_1A_2^t+A_2B_1+B_2A_1^t$ is symmetric. Moreover, we claim that the image lies in fact in the space of diagonal zero symmetric matrices. Indeed,
$$F_{A_1,A_2}(B_1,B_2) = A_1B_2+B_1A_2^t+A_2B_1+B_2A_1^t = A_1B_2+(A_1B_2)^t + B_1A_2^t + (B_1A_2^t)^t$$
and in characteristic $2$ we have $diag(M+M^t)=0$ for any $M$, so it follows that $diag(F_{A_1,A_2}(B_1,B_2))=0$. Therefore, 
$$F_{A_1,A_2}(B_1,B_2) \in S=\{ M \in \Sym_n \mid diag(M)=0\},$$
whose dimension is $\dim(S)=\frac{n(n-1)}{2}$.

Next, let us consider the morphism $\phi: \mathcal{C}_2^{\text{nil}}(\mathfrak{p}) \twoheadrightarrow \mathcal{C}_2^{\text{nil}}(\mathfrak{gl}_n)$ given by 
$$\Bigg(x=\begin{pmatrix} A_1 & B_1 \\ 0 & A_1^t \end{pmatrix}, y=\begin{pmatrix} A_2 & B_2 \\ 0 & A_2^t \end{pmatrix} \Bigg) \mapsto (A_1,A_2).$$
By classical results of Baranovsky \cite{Baranovsky} and Premet \cite{Premet}, $\mathcal{C}_2^{\text{nil}}(\mathfrak{gl}_n)$ is irreducible of dimension $n^2-1$. Now fix a point in the codomain $(A_1,A_2) \in \mathcal{C}_2^{\text{nil}}(\mathfrak{gl}_n)$. Its fiber $\phi^{-1}(A_1,A_2)$ consists precisely of the solutions $(B_1,B_2) \in \Sym_n \times \Sym_n$ of the equation $F_{A_1,A_2}(B_1,B_2)=0$, namely the kernel of the linear map $F_{A_1,A_2}:\Sym_n \times \Sym_n \rightarrow S$. Thus, $\phi^{-1}(A_1,A_2)$ is an affine subspace whose dimension is
$$\dim(\phi^{-1}(A_1,A_2))=\dim \ker(F_{A_1,A_2}) = \dim(\Sym_n \times \Sym_n) - \rank(F_{A_1,A_2}).$$
Now, $\dim(\Sym_n \times \Sym_n)=2\frac{n(n+1)}{2}=n(n+1)$ and $F_{A_1,A_2}$ takes values in $S$, so $\dim(\rank(F_{A_1,A_2})) \leq \dim(S)=\frac{n(n-1)}{2}$. Hence, for any $(A_1,A_2)$,
$$\dim(\phi^{-1}(A_1,A_2)) \geq n(n+1)-\frac{n(n-1)}{2} = \frac{n(n+3)}{2}.$$
Next, applying the fibre dimension theorem to $\phi$, we obtain that for any irreducible component $Y$ of $\mathcal{C}_2^{\text{nil}}(\mathfrak{p})$,
$$\dim(Y) \geq \dim(\mathcal{C}_2^{\text{nil}}(\mathfrak{gl}_n)) +  \frac{n(n+3)}{2}= (n^2-1)+\frac{n(n+3)}{2}.$$
Now let $X$ be an irreducible component of $\mathcal{C}_2^{\text{nil}}(\mathfrak{sp}_{2n})$, $Y$ an irreducible component of $\mathcal{C}_2^{\text{nil}}(\mathfrak{p})$ and let
$$Q=\biggl\{ \begin{pmatrix} I & 0 \\ \alpha & I \end{pmatrix} \Biggl\vert \alpha \text{ invertible} \biggr\} \subseteq Sp_{2n}$$
and consider the morphism $\psi:Q \times Y \rightarrow X$ given by conjugation
$$(g,(y_1,y_2)) \mapsto (gy_1g^{-1},gy_2g^{-1}).$$

\begin{lemma}
    The map $\psi$ is generically finite.
\end{lemma}

\begin{proof}
    Let $\Bigg( y_1=\begin{pmatrix} A & A' \\ 0 & A^t \end{pmatrix}, y_2=\begin{pmatrix} B & B' \\ 0 & B^t \end{pmatrix} \Bigg)\in Y$ and $q=\begin{pmatrix} I & 0 \\ \alpha & I \end{pmatrix} \in Q$. A computation shows that
    $$\psi(g,(y_1,y_2)) = \Bigg( \begin{pmatrix} A+A' \alpha & A' \\ \alpha A +\alpha A' \alpha+A^t \alpha & (A+A' \alpha)^t \end{pmatrix}, \begin{pmatrix} B+B' \alpha & B' \\ \alpha B +\alpha B' \alpha+B^t \alpha & (B+B' \alpha)^t \end{pmatrix} \Bigg).$$
    Thus, $g$ centralizes the pair $(y_1,y_2)$ if and only if the following relations hold:
    \[ A'\alpha=0
    \]
    \[
        \alpha A + \alpha A' \alpha + A^t \alpha=0 \implies \alpha A+A^t \alpha=0 \implies \alpha A \text{ symmetric}
    \]
    \[
        B'\alpha=0
    \]
    \[
        \alpha B + \alpha B' \alpha + B^t \alpha=0 \implies \alpha B+B^t \alpha=0\implies \alpha B \text{ symmetric}
    \]
    Now let us consider a generic point $(y_1,y_2) \in Y$. One can take $A',B'$ invertible, which is an open condition. Let us also show that the set of such elements is non-empty. We have our element $y_1=\begin{pmatrix} A & A' \\ 0 & A^t \end{pmatrix}$. First, there exists $S \in GL_n(k)$ such that $S^{-1}AS=A^t$. Moreover, one can choose $S$ to be symmetric (or conjugate $A$ to be symmetric and then take $S=I$). This is a result of Taussky-Zassenhaus \cite{Taussky-Zassenhaus} for which we will provide a different proof ...... Then the element $N=\begin{pmatrix} 0 & S \\ 0 & 0 \end{pmatrix} \in \mathfrak{sp}_{2n}$ commutes with $y_1$ since $AS=SA^t$. Therefore, $N$ also commutes with $y_1$ up to changing the upper-right block, namely it commutes with anything of the form $y_1+\lambda N = \begin{pmatrix} A & A'+\lambda S \\ 0 & A^t \end{pmatrix}$. Next, consider the polynomial $f(\lambda)=\det(A'+ \lambda S)$. This has degree $n$ and its leading term is $\det(\lambda S)=\lambda^n \det(S)$, which is not identically $0$ since $S$ is invertible. Hence $f(\lambda)\neq 0$ for all but finitely many values of $\lambda$, so for generic $\lambda$, $A'+\lambda S$ is invertible. Notice that the exact same argument works for $B'$.

    By the above, there is an open dense subset $Y' \subset Y$ where $A',B'$ are invertible. Using the above equations, we see that we must have $\alpha=0$, so the centralizer of any pair $(y_1,y_2) \in Y'$ must be trivial. Therefore, over the open subset $\psi(Q \times Y') \subset X$, each fibre is a single point, so $\psi$ is generically finite.
\end{proof}

\begin{theorem}[Taussky-Zassenhaus]
    Let $k$ be a quadratically closed field of characteristic $2$ and let $A \in M_n(k)$. Then there exists $S \in GL_n(k)$ with $S$ symmetric such that $S^{-1}AS=A^t$.
\end{theorem}

\begin{proof}
    First of all, there exist symmetric matrices $S,T \in M_n(k)$ with $S$ invertible such that $A=ST$. Next, since $k$ is quadratically closed, every invertible symmetric matrix over $k$ is a Gram matrix (equivalently, over a quadratically closed field of characteristic $2$, every nondegenerate symmetric bilinear form is congruent to the standard one). Thus, we may write $S=XX^t$ for some $X \in GL_n(k)$. So we have $A=ST=XX^tT$. Conjugating this by $X$ we obtain 
    $$X^{-1}AX=X^{-1}(XX^tT)X=X^tTX.$$
    Set $B=X^tTX$ and note that $$B^t=(X^tTX)^t=X^tT^tX=X^tTX=B,$$
    so $B$ is symmetric. Now, since $A=ST$, we note that
    $$A^t=(ST)^t=T^tS^t=TS$$ and then conjugating $A$ by $S$,
    $$S^{-1}AS=S^{-1}(ST)S = TS=A^t,$$
    so $S$ conjugates $A$ to $A^t$ and $S$ is symmetric and invertible by construction.
\end{proof}

\begin{lemma}
     Let $k$ be a quadratically closed field of characteristic $2$ and let $A \in M_n(k)$. Then there exists $S \in GL_n(K)$ such that $S^{-1}AS$ is symmetric.
\end{lemma}
\begin{proof}
    First of all, there exist symmetric matrices $S,T \in M_n(k)$ with $S$ invertible such that $A=ST$. Next, we may write $S=XX^t$ for some $X \in GL_n(k)$. So we have $A=ST=XX^tT$. Conjugating this by $X$ we obtain 
    $$X^{-1}AX=X^{-1}(XX^tT)X=X^tTX.$$
    Set $B=X^tTX$ and note that $$B^t=(X^tTX)^t=X^tT^tX=X^tTX=B,$$
    so $B$ is symmetric.
\end{proof}

By the above, we have $\dim(X)=\dim(Q \times Y)$ and $X$ is the closure of $\psi(Q \times Y)$. Now, by the discrepancy argument, we have $\dim(X) \leq \dim(\mathfrak{sp}_{2n}) + n -1$ (MAYBE SAY A BIT MORE HERE).
On the other hand, we have a lower bound given by
\begin{align*}
    \dim(X) &= \dim(Q)+\dim(Y) \\
    &\geq \frac{n(n+1)}{2} + \bigg( (n^2-1)+\frac{n(n+3)}{2} \bigg)\\
    &= 2n^2+2n-1 \\
    &= \dim(\mathfrak{sp}_{2n}) +n-1.
\end{align*}

\printbibliography

@book{liebeck-seitz,
  title={Unipotent and Nilpotent Classes in Simple Algebraic Groups and Lie Algebras},
  author={Liebeck, M.W. and Seitz, G.M.},
  isbn={9780821869208},
  lccn={2011043518},
  series={Mathematical surveys and monographs},
  url={https://books.google.ro/books?id=Th-CAwAAQBAJ},
  year={2012},
  publisher={American Mathematical Society}
}

@article{Motzkin-Taussky,
   author = {Motzkin, T. and Taussky, O.},
    title = "{Pairs of matrices with property L. II}",
  journal = {Trans. Amer. Math. Soc. 80, 387-401},
     year = 1955
}

@article{Richardson,
author = {Richardson, R. W.},
journal = {Compositio Mathematica},
language = {eng},
number = {3},
pages = {311-327},
publisher = {Sijthoff et Noordhoff International Publishers},
title = {Commuting varieties of semisimple Lie algebras and algebraic groups},
url = {http://eudml.org/doc/89407},
volume = {38},
year = {1979},
}

@article{Premet,
author = {Premet, Alexander},
year = {2003},
month = {03},
pages = {},
title = {Nilpotent commuting varieties of reductive Lie algebras},
volume = {154},
journal = {Inventiones Mathematicae},
doi = {10.1007/s00222-003-0315-6}
}

@article{Levy,
title = {Commuting Varieties of Lie Algebras over Fields of Prime Characteristic},
journal = {Journal of Algebra},
volume = {250},
number = {2},
pages = {473-484},
year = {2002},
issn = {0021-8693},
doi = {https://doi.org/10.1006/jabr.2001.9083},
url = {https://www.sciencedirect.com/science/article/pii/S0021869301990830},
author = {Paul Levy}
}

@article{mypaper,
title = {The commuting variety of pgln},
journal = {Journal of Algebra},
volume = {665},
pages = {229-242},
year = {2025},
issn = {0021-8693},
url = {https://www.sciencedirect.com/science/article/pii/S0021869324005945},
author = {Vlad Roman},
}

@article{Baranovsky,
title = {The variety of pairs of commuting nilpotent matrices is irreducible},
journal = {Transformation Groups},
volume = {6},
year = {2001},
doi = {https://doi.org/10.1007/BF01236059},
author = {Baranovsky, V.}
}

@article{Taussky-Zassenhaus,
  title={On the similarity transformation between a matirx and its transpose},
  author={Olga Taussky and Hans Zassenhaus},
  journal={Pacific Journal of Mathematics},
  year={1959},
  volume={9},
  pages={893-896},
  url={https://api.semanticscholar.org/CorpusID:119722390}
}
\end{document}